\documentclass[11pt]{amsart}
\usepackage{amsthm,amsmath,amssymb,enumerate}
\usepackage{graphicx,color}
\usepackage{hyperref}
\title{Bergman iteration and $C^{\infty}$-convergence
towards K\"ahler-Ricci flow}
\author{Ryosuke Takahashi}

\address{Mathematical Institute, Tohoku University, 6-3, Aoba, Aramaki, Aoba-ku, Sendai, 980-8578, Japan}

\email{ryosuke.takahashi.a7@tohoku.ac.jp}


\thanks{This work was supported by Grant-in-Aid for JSPS Fellows Number 16J01211.}
\makeatletter

\@addtoreset{equation}{section}
\makeatother

{
\theoremstyle{definition}
\newtheorem{definition}{Definition}[section]
\newtheorem*{acknowledgements}{Acknowledgements}
\newtheorem*{claim}{Claim}
}

{
\theoremstyle{plain}
\newtheorem{theorem}{Theorem}[section]
\newtheorem{proposition}{Proposition}[section]
\newtheorem{lemma}{Lemma}[section]

}

{
\theoremstyle{remark}
\newtheorem{remark}{Remark}[section]

}
\begin{document}
\begin{abstract}
On a polarized manifold $(X,L)$, the Bergman iteration $\phi_k^{(m)}$ is defined as a sequence of Bergman metrics on $L$ with two integer parameters $k, m$. We study the relation between the K\"ahler-Ricci flow $\phi_t$ at any time $t \geq 0$ and the limiting behavior of metrics $\phi_k^{(m)}$ when $m=m(k)$ and the ratio $m/k$ approaches to $t$ as $k \to \infty$. Mainly, three settings are investigated: the case when $L$ is a general polarization on a Calabi-Yau manifold $X$ and the case when $L=\pm K_X$ is the (anti-) canonical bundle. 
Recently, Berman showed that the convergence $\phi_k^{(m)} \to \phi_t$ holds in the $C^0$-topology, in particular, the convergence of curvatures holds in terms of currents. In this paper, we extend Berman's result and show that this convergence actually holds in the smooth topology.
\end{abstract}
\maketitle
\tableofcontents
\section{Introduction}
\label{sec1}
\subsection{Background}
Throughout this paper, let $(X,L)$ be an $n$-dimensional polarized manifold (i.e., $X$ is a compact K\"ahler manifold with an ample line bundle $L$), and ${\mathcal H}(L)$ is the space of smooth plurisubharmonic weights with strictly positive curvature (where the term ``weight'' is an additive notation for hermitian metrics, for instance, see \cite{BB10}). For $\phi \in {\mathcal H}(L)$, $\omega_{\phi}$ denotes the curvature locally represented as
\[
\omega_{\phi}=\frac{\sqrt{-1}}{2 \pi} \partial \bar{\partial} \phi \in c_1(L).
\]
For simplicity, we may assume that $c_1(L)^n = n!$, i.e., the Monge-Amp\`ere volume form
\[
{\rm MA}(\phi):=\frac{\omega_{\phi}^n}{n!}
\]
has unit volume. Let $\mu$ be a map from ${\mathcal H}(L)$ to the space of all smooth volume forms on $X$. We assume that $\mu=\mu(\phi)$ depends smoothly on $\phi$. Given an initial metric $\phi_0 \in {\mathcal H}(L)$, {\it K\"ahler-Ricci flow} in ${\mathcal H}(L)$ is defined by the Monge-Amp\`ere evolution equation
\begin{equation} \label{krf}
\frac{\partial{\phi_t}}{\partial t} = \log \frac{{\rm MA}(\phi_t)}{\mu(\phi_t)}.
\end{equation}
The K\"ahler-Ricci flow is an analytic study object in nature, whereas Berman \cite{Ber13} proposed a numerical algorithm to study \eqref{krf} called ``Bergman iteration'': for any integer $k$, let ${\mathcal B}_k$ be the space of hermitian forms on $H^0(X,kL)$. We associate the pair $(\phi, \mu(\phi))$ to a {\it Hilbert map}:
\[
{\rm Hilb}_{k, \mu}(\phi) \colon {\mathcal H}(L) \longrightarrow {\mathcal B}_k
\]
defined by
\[
\|s\|_{{\rm Hilb}_{k, \mu}(\phi)} ^2 = \int_X |s|^2 e^{-k \phi} \mu(\phi)
\]
for $s \in H^0(X,kL)$. Conversely, {\it Fubini-Study map}:
\[
{\rm FS}_k \colon {\mathcal B}_k \longrightarrow {\mathcal H}(L)
\]
is defined by
\[
{\rm FS}_k (H)=\frac{1}{k} \log \left( \frac{1}{N_k} \sum_{i=1}^{N_k} |s_i|^2 \right)
\]
for $H \in {\mathcal B}_k$, where $(s_i)$ is {\it any} orthonormal basis with respect to $H$ and
\[
N_k:= \dim H^0(X, kL).
\]
Notice that the definition of ${\rm FS}_k (H)$ does not depend on the specific choice of $(s_i)$. In fact, ${\rm FS}_k (H)$ is, to the letter, just the restriction to $X$ of the Fubini-Study weight determined by $H$.
An element in ${\mathcal B}_k$, or in the image of injective map ${\rm FS}_k$ is called a {\it Bergman metric} (at level $k$). 
Let ${\mathcal T}_{k, \mu}$ be the composition of these two maps
\[
{\mathcal T}_{k, \mu}:={\rm FS}_k \circ {\rm Hilb}_{k, \mu}.
\]
We consider the sequence of Bergman metrics:
\begin{equation} \label{bit}
\phi_k^{(m)}:=({\mathcal T}_{k,\mu})^m (\phi_0).
\end{equation}
The sequence $\phi_k^{(m)}$ is called the {\it Bergman iteration} (starting at $\phi_0$). We mainly consider three cases $({\bf S}_0)$ and $({\bf S}_{\pm})$. Then the stational points $\phi_{\rm KE}$ of the K\"ahler-Ricci flow, i.e., the solutions to the Monge-Amp\`ere equation
\begin{equation} \label{mae}
{\rm MA}(\phi_{\rm KE}) = \mu(\phi_{\rm KE})
\end{equation}
is a {\it K\"ahler-Einstein weight} as we see below.

\noindent
{\it The Calabi-Yau setting } $({\bf S}_0)$.
Let $X$ be a Calabi-Yau manifold (i.e., the canonical bundle $K_X$ is holomorphically trivial) and $\Omega$ a global holomorphic $n$-form that vanishes nowhere. Then
\begin{equation} \label{muz}
\mu_0:=C (\sqrt{-1})^{n^2} \Omega \wedge \bar{\Omega}
\end{equation}
defines a smooth volume form, where $C$ is a normalizing constant so that $\mu_0$ has unit volume. We set $\mu:=\mu_0$. Since $\Omega$ is holomorphic, differentiating \eqref{mae} yields
\[
{\rm Ric} (\omega_{\phi_{\rm KE}}) = - \frac{\sqrt{-1}}{2 \pi} \partial \bar{\partial} \log {\rm MA}(\phi_{\rm KE})
= - \frac{\sqrt{-1}}{2 \pi} \partial \bar{\partial} \log \mu_0
= 0.
\]
Thus the solution corresponds to a Ricci-flat K\"ahler metric. Berman \cite[Theorem 3.1]{Ber13} (also see \cite{Cao85}) showed that there exists a long time solution $\phi_t$ of \eqref{krf}, and $\phi_t$ converges to the K\"ahler-Einstein weight $\phi_{\rm KE}$ in the $C^{\infty}$-topology as $t \to \infty$.

\vspace{4mm}

\noindent
{\it The (anti-) canonical setting } $({\bf S}_{\pm})$. We consider the case when the ample line bundle $L$ is the (anti-) canonical bundle $\pm K_X$ and $\mu$ is the canonical volume form
\[
\mu_{\pm}(\phi):=e^{\pm \phi}.
\]
Then differentiating the equation \eqref{mae}, we have
\[
{\rm Ric}(\omega_{\phi_{\rm KE}})= - \frac{\sqrt{-1}}{2 \pi} \partial \bar{\partial} \log {\rm MA}(\phi_{\rm KE})
= - \frac{\sqrt{-1}}{2 \pi} \partial \bar{\partial} \log \mu_{\pm}(\phi_{\rm KE})= \mp \omega_{\phi_{\rm KE}}.
\]
Thus the solution corresponds to a K\"ahler-Einstein metric of negative (resp. positive) scalar curvature. In the both settings, the long time solution of \eqref{krf} always exists. Moreover, if we consider the normalized canonical volume form
\[
\bar{\mu}_{\pm}(\phi):=\frac{\mu_{\pm}(\phi)}{\int_X \mu_{\pm}(\phi)}
\]
instead of $\mu_{\pm}$, the normalized K\"ahler-Ricci flow converges to a K\"ahler-Einstein weight in the setting $L=K_X$. On the other hand, we need some extra assumptions to prove the same result in the setting $L=-K_X$ (see \cite[Theorem 4.1]{Ber13} and \cite{TZ07}).

\vspace{4mm}

Let $m=m(k)$ be a sequence of non-negative integers such that the ratio $m/k$ converges to some real number $t \geq 0$ as $k \to \infty$. We would like to call such a limit the {\it double scaling limit} and simply write as $m/k \to t$.
In the seminal paper of Berman \cite{Ber13}, he showed that in each of three settings $({\bf S}_0)$ and $({\bf S}_{\pm})$, in the double scaling limit $m/k \to t$, the Bergman iteration $\phi_k^{(m)}$ converges to the K\"ahler-Ricci flow $\phi_t$ in the $C^0$-topology. In particular, the convergence of K\"ahler metrics $\omega_{\phi_k^{(m)}} \to \omega_{\phi_t}$ holds in terms of current.
\subsection{The main result}
As expected in \cite{Ber13}, the statement of Berman's theorem still holds in a stronger sense, that is, we can show the following:
\begin{theorem} \label{thm}
Let $(X,L)$ be a polarized manifold. For any $\phi_0 \in {\mathcal H}(L)$, let $\phi_k^{(m)}$ be the Bergman iteration \eqref{bit} and $\phi_t$ the K\"ahler-Ricci flow \eqref{krf} starting at the same initial weight $\phi_0$. Then, in each of three settings $({\bf S}_0)$, $({\bf S}_{\pm})$, in the double scaling limit $m/k \to t$, we have the convergence of K\"ahler metrics
\[
\omega_{\phi_k^{(m)}} \to \omega_{\phi_t}
\]
in the $C^{\infty}$-topology on $X$. More precisely, for any non-negative integer $l$, we have
\[
\left\|\omega_{\phi_k^{(m)}} - \omega_{\phi_t}\right\|_{C^l}=O(k^{-1}).
\]
\end{theorem}
\begin{remark}
In the proof of Theorem \ref{thm} for the setting $({\bf S}_0)$, we do not use the Calabi-Yau hypothesis and the equation \eqref{muz}. Therefore, Theorem \ref{thm} can be extended to the case when $(X,L)$ is a general polarized manifold and $\mu$ is any fixed smooth volume form with unit volume (where the long time solution and convergence of the K\"ahler-Ricci flow for this case were obtained also in \cite[Theorem 3.1]{Ber13} or \cite{Cao85}).

\end{remark}
Our problem is similar to the problem discussed in many places (e.g. \cite{Don02}, \cite{Fin10}, \cite{Has15}, etc.). The key idea to prove the main theorem is constructing the higher order approximation towards the Bergman iteration by adding polynomials of $k^{-1}$ with coefficient functions $\eta_i(t)$ to $\phi_t$:
\[
\widetilde{\phi}_t^{(r)}:= \phi_t + \sum_{i=1}^{r} k^{-i} \eta_i (t).
\]
We can find a successful choice of $\eta_i$ as a solution of a heat equation and kill the lower order terms appearing in the distance $d \left( \phi_k^{(m)},\widetilde{\phi}_{m/k}^{(r)} \right)= \sup_X \left| \phi_k^{(m)} - \widetilde{\phi}_{m/k}^{(r)} \right|$, which overcomes the polynomial growth of the distance functions on ${\mathcal B}_k$ (cf. Lemma \ref{bdf}). The $C^l$-norms for Bergman metrics are controlled by the upper bound of the operator norm and the distance function provided the family of metrics has bounded geometry (cf. Lemma \ref{ccp}). These projective and analytic estimates were established in \cite{Don02} and \cite{Fin10}, and are widely used throughout this paper.

The author expects that in the case when $m/k \to \infty$, we can show the $C^{\infty}$-convergence $\omega_{\phi_k^{(m)}} \to  \omega_{\phi_{\rm KE}}$ as long as $m/k$ has a polynomial growth of $k$, under the $C^{\infty}$-convergence assumption $\phi_t \to \phi_{\rm KE}$. Then the sequence $\phi_k^{(m)}$ gives a dynamical construction of solutions $\phi_{\rm KE}$ to the Monge-Amp\`ere equation \eqref{mae}.  We can prove this if we have the uniform control of the higher order derivatives for the functions $\eta_1, \ldots, \eta_r$. However it is hard to prove this in general, so that we leave this problem for the future.

\subsection{Relation with other results}
If we take $m=1$ in Theorem \ref{thm}, we obtain the $C^{\infty}$-convergence of K\"ahler metrics $\omega_{\phi_k^{(1)}} \to \omega_{\phi_0}$ as raising the exponent $k \to \infty$. Thus our main theorem can be seen as an extension of Tian's result \cite{Tia90} along the K\"ahler-Ricci flow (also see Proposition \ref{abm}). On the other hand, we consider the discrete time limit $m \to \infty$ with a fixed exponent $k$. Then one would expect that the sequence $\phi_k^{(m)}$ converges to a balanced weight $\phi_{{\rm bal}, k} \in {\mathcal B}_k$, i.e., a fixed point of the iteration map:
\[
{\mathcal T}_{k, \mu}(\phi_{{\rm bal}, k})=\phi_{{\rm bal}, k}.
\]
This is actually true in the setting $({\bf S}_0)$ and $(\bf{S}_{\mu_{+}})$ (cf. \cite[Theorem 3.9, Theorem 4.14]{Ber13}), whereas the same result generally does not hold in the setting $(\bf{S}_{\mu_{-}})$ because of the absence of balanced weights (cf. \cite[Example 4.3, Example 5.6]{ST16}). The large $k$ limit of balanced weights $\phi_{{\rm bal}, k}$ are also studied deeply in \cite[Theorem 7.1]{BBGZ13}.
\begin{acknowledgements}
The author would like to express his gratitude to his advisor Professor Shigetoshi Bando and Professor Ryoichi Kobayashi for useful discussions on this article. The author also would like to thank Professor Shin Kikuta, Satoshi Nakamura and Yusuke Miura for several helpful comments. This research is supported by Grant-in-Aid for JSPS Fellows Number 16J01211.
\end{acknowledgements}
\section{Estimates} \label{sec2}
\subsection{The $C^0$-estimate}
\subsubsection{Large $k$ asymptotics of Bergman functions}
Let $(X,L)$ be a polarized manifold and $\mu=\mu(\phi)$ a smooth volume form depending smoothly on $\phi \in {\mathcal H}(L)$. We define the {\it Bergman function} associated to $(\phi, \mu(\phi))$ by
\[
\rho_{k, \mu} (\phi):=\sum_{i=1}^{N_k} |s_i|^2e^{-k\phi},
\]
where $(s_i)$ is any ${\rm Hilb}_{k, \mu} (\phi)$-orthonormal basis. Then it is not hard to see that the function $\rho_{k, \mu}(\phi)$ does not depend on the choice of $(s_i)$. We introduce the notion of {\it normalized} Bergman function
\[
\bar{\rho}_{k, \mu} (\phi):= \frac{1}{N_k} \rho_{k, \mu} (\phi)
\]
so that $\int_X \bar{\rho}_{k, \mu} (\phi) \mu(\phi)=1$, and put
\[
F_{\mu}^{(k)}(\phi):=\frac{1}{k} \log \bar{\rho}_{k,\mu}(\phi).
\]
Then we have
\begin{equation} \label{cim}
{\mathcal T}_{k,\mu} - {\rm Id} = F_{\mu}^{(k)}.
\end{equation}
In particular, when we take $\mu$ as the Monge-Amp\`ere volume form ${\rm MA}(\phi)$, we drop the notation of $\mu$ and simply write ${\rm Hilb}_k$, $\rho_k$, ${\mathcal T}_k$, etc.

Now we recall the property of Bergman functions essentially obtained by Bouche \cite{Bou90}, Catlin \cite{Cat99}, Tian \cite{Tia90} and Zelditch \cite{Zel98}. The asymptotic expansions of the Bergman function associated to the space of global sections of $kL + {\mathbb C}$ is also studied in \cite[Theorem 1.1, Theorem 1.3]{DLM06} and \cite[Section 2.5]{BBS08}, where  ${\mathbb C}$ denotes the trivial line bundle with the hermitian metric $\mu(\phi)/{\rm MA}(\phi)$.
\begin{proposition} \label{abf}
We have the following asymptotic expansion of Bergman function:
\[
\rho_{k, \mu} (\phi) = (b_0 k^n + b_1 k^{n-1} + b_2 k^{n-2} + \cdots) \cdot \frac{{\rm MA}(\phi)}{\mu(\phi)},
\]
Each coefficient $b_i$ can be written as a polynomial in the Riemannian curvature ${\rm Riem}(\omega_{\phi})$, the curvature of $\mu(\phi)/{\rm MA}(\phi)$, their derivatives and contractions with respect to $\omega_{\phi}$. In particular, we have
\[
b_0=1, \;\;\;, b_1= \frac{1}{2} S(\omega_{\phi})- \Delta_{\phi} \log \frac{{\rm MA}(\phi)}{\mu(\phi)},
\]
where $S(\omega_{\phi})$ is the scalar curvature of $\omega_{\phi}$ and $\Delta_{\phi}$ is the (negative) $\bar{\partial}$-Laplacian with respect to $\omega_{\phi}$. The above expansion is uniform as long as $\phi$ stays in a bounded set in the $C^{\infty}$-topology. More precisely, for any integer $p$ and $l$, there exists a constant $C_{p, l}$ such that
\[
\left\| \rho_{k, \mu} (\phi) - \sum_{i=0}^p b_i k^{n-i} \right\|_{C^l} < C_{p,l} \cdot k^{n-p-1}.
\]
We can take the constant $C_{p, l}$ independently of $\phi$ as long as $\phi$ stays in a bounded set in the $C^{\infty}$-topology.
\end{proposition}
By the Riemann-Roch formula, we find that $N_k$ is a polynomial of $k$ of degree $n$:
\[
N_k=k^n+\frac{1}{2} \bar{S} k^{n-1} + O(k^{n-2}),
\]
where we denote the average of scalar curvature by $\bar{S}$ (which is independent of a choice of $\phi \in {\mathcal H}(L)$). Combining with Proposition \ref{abf}, we also have a uniform asymptotic expansion of a normalized Bergman function: 
\[
\bar{\rho}_{k, \mu} (\phi) = (\bar{b}_0 + \bar{b}_1 k^{-1} + \bar{b}_2 k^{-2} + \cdots) \cdot \frac{{\rm MA}(\phi)}{\mu(\phi)},
\]
where $\bar{b}_0=1$ and $\bar{b}_1= \frac{1}{2} (S(\omega_{\phi})-\bar{S}) - \Delta_{\phi} \log \frac{{\rm MA}(\phi)}{\mu(\phi)}$.
\subsubsection{Higher order approximation}
In what follows, we consider only three cases $({\bf S}_0)$, $({\bf S}_{\pm})$ introduced in Section \ref{sec1}. The following properties follow directly from the definition of $F_{\mu}^{(k)}$:
\begin{proposition}\label{pfs}
We have the following properties in each settings:
\begin{enumerate}
\item In the setting $({\bf S}_0)$, the function $F_{\mu_0}^{(k)}$ satisfies
\[
F_{\mu_0}^{(k)}(\phi+c)=F_{\mu_0}^{(k)}(\phi)
\]
for any $c \in {\mathbb R}$ and $\phi \in {\mathcal H}(L)$.
\item In the setting $({\bf S}_{\pm})$, the function $F_{\mu_{\pm}}^{(k)}$ satisfies
\[
F_{\mu_{\pm}}^{(k)}(\phi+c)=F_{\mu_{\pm}}^{(k)}(\phi) \mp \frac{c}{k}
\]
for any $c \in {\mathbb R}$ and $\phi \in {\mathcal H}(L)$.
\end{enumerate}
\end{proposition}
Let $d$ be the distance function defined by the sup-norm
\[
d(\phi, \psi):=\sup_X |\phi-\psi|
\]
for $\phi, \; \psi \in {\mathcal H}(L)$.
We also use the monotonicity for the iteration map ${\mathcal T}_{k,\mu}$:
\begin{proposition}[\cite{Ber13}, Proposition 3.13 and Proposition 4.12] \label{mit}
In each of three settings $({\bf S}_0)$, $({\bf S}_{\pm})$, the monotonicity for the iteration map holds, i.e., for any weights $\phi, \psi \in {\mathcal H}(L)$ such that $\phi \leq \psi$, we have ${\mathcal T}_{k,\mu}(\phi) \leq {\mathcal T}_{k,\mu}(\psi)$. Moreover, we have the following:
\begin{enumerate}
\item
In the setting $({\bf S}_0)$, an inequality
\[
d({\mathcal T}_{k,\mu_0}(\phi),{\mathcal T}_{k,\mu_0}(\psi)) \leq d(\phi, \psi)
\]
holds for any $\phi, \psi \in {\mathcal H}(L)$.
\item
In the setting $({\bf S}_{\pm})$, an inequality
\[
d({\mathcal T}_{k,\mu_{\pm}}(\phi),{\mathcal T}_{k,\mu_{\pm}}(\psi)) \leq \left(1 \mp \frac{1}{k} \right) d(\phi, \psi)
\]
holds for any $\phi, \psi \in {\mathcal H}(L)$.
\end{enumerate}
\end{proposition}
\begin{proof}
For the sake of exposition and completeness, we shall provide a complete proof. In each of three settings, the following characterization for Bergman function holds (cf. \cite[Lemma 6.2]{Sze14}):
\[
\rho_{k,\mu}(\phi)(x)=\sup_{s \in H^0(X,kL)} \frac{|s(x)|^2e^{-k\phi}}{\int_X |s|^2e^{-k\phi} \mu(\phi)}.
\]
Since $\phi \leq \psi$, we have
\[
\int_X |s|^2e^{-k\phi} \mu(\phi) \geq \int_X |s|^2e^{-k\psi} \mu(\psi)
\]
for any $s \in H^0(X,kL)$. Hence, by \eqref{cim}, we obtain
\[
{\mathcal T}_{k,\mu}(\phi)=\phi+F_{\mu}^{(k)}(\phi) \leq \psi + F_{\mu}^{(k)}(\psi)={\mathcal T}_{k,\mu}(\psi),
\]
which proves the first statement. Next, we consider (1). If we set $C:=d(\phi,\psi)$, we have
\[
\phi \leq \psi + C, \;\; \psi \leq \phi + C.
\]
Applying the map ${\mathcal T}_{k,\mu_0}$ to the first equation yields
\begin{eqnarray*}
{\mathcal T}_{k,\mu_0}(\phi) &\leq& {\mathcal T}_{k,\mu_0}(\psi+C) \;\;\text{(the monotonicity for ${\mathcal T}_{k,\mu_0}$)}\\
&=& \psi+C +F_{\mu_0}^{(k)}(\psi+C) \;\;\text{(by \eqref{cim})}\\
&=& \psi+C+F_{\mu_0}^{(k)}(\psi) \;\;\text{(by Proposition \ref{pfs})}\\
&=& {\mathcal T}_{k,\mu_0}(\psi)+C \;\;\text{(by \eqref{cim})}.
\end{eqnarray*}
Applying the map ${\mathcal T}_{k,\mu_0}$ to the second equation yields another inequality ${\mathcal T}_{k,\mu_0}(\psi) \leq {\mathcal T}_{k,\mu_0}(\phi)+C$, and thus we obtain (1). A similar proof also works for (2).

\end{proof}
Let $\phi_t$ be a solution of the K\"ahler-Ricci flow \eqref{krf}. We perturb $\phi_t$ as
\[
\widetilde{\phi}_t^{(r)}:= \phi_t + \sum_{i=1}^{r} k^{-i} \eta_i (t),
\]
where $\eta_1(t), \ldots, \eta_r(t)$ are smooth functions on $X \times [0, \infty)$. Then $\widetilde{\phi}_t^{(r)} \in {\mathcal H}(L)$ for sufficiently large $k$, and $\widetilde{\phi}_t^{(r)} \to \phi_t$ in the $C^{\infty}$-topology as $k \to \infty$. In what follows, let $T>0$ be a large constant, and all $O$ are meant to hold uniformly for $t \leq T$ as $k \to \infty$. For instance, we have:
\begin{remark} \label{frm}
$\widetilde{\phi}_t^{(r)} - \phi_t=O(k^{-1})$. 
\end{remark}
The following is an analogue of \cite[Theorem 11]{Fin10}. 
\begin{lemma} \label{hoa}
Let $r$ be any non-negative integer. Then there exists an appropriate choice of $\eta_1, \ldots, \eta_r$ (depending only on the initial data $\phi_0$) such that
\begin{equation} \label{hok}
d \left( \phi_k^{(m)},\widetilde{\phi}_{m/k}^{(r)} \right) \leq C \cdot k^{-(r+1)}
\end{equation}
holds for any pair $(k,m)$ such that $m/k \leq T$,
where the constant $C>0$ depends only on $T$.
\end{lemma}
\begin{proof}
We first show the following claim.

\vspace{2mm}

\begin{claim} For any $t \in [0, T]$, there exists an appropriate choice of $\eta_1, \ldots, \eta_r$ such that
\begin{equation} \label{clm}
\sup_X \left| (\widetilde{\phi}_{t+1/k}^{(r)} - \widetilde{\phi}_t^{(r)}) - F_{\mu}^{(k)}(\widetilde{\phi}_t^{(r)}) \right| \leq \frac{C}{k^{r+2}},
\end{equation}
where the constant $C>0$ depends only on $T$.
\end{claim}
\begin{proof}[The proof of Claim]
We prove this by induction of $r$. In the case when $r=0$, the equation \eqref{clm} follows from the proof of \cite[Theorem 3.15, Theorem 4.18]{Ber13}.

\vspace{2mm}

\noindent
{\it The setting} $({\bf S}_0)$. We assume that the claim holds for some appropriate choice of $\eta_1,\ldots,\eta_r$. First, for a heuristic argument, let $\eta_{r+1}$ be any smooth function on $X \times [0, \infty)$. By the mean value theorem, we can compute
\begin{eqnarray*}
\widetilde{\phi}_{t+1/k}^{(r+1)} - \widetilde{\phi}_t^{(r+1)} &=& \phi_{t+1/k} - \phi_t + \sum_{i=1}^{r+1} k^{-i} ( \eta_i (t+1/k) - \eta_i (t)) \\
&=& \frac{1}{k} \cdot \frac{\partial \phi_t}{\partial t} + \sum_{i=1}^{r} k^{-(i+1)} M_i (t) \\
&\hbox{}& + k^{-(r+2)} \left( \frac{\partial \eta_{r+1}}{\partial t} (t) + G_r (t) \right) + O(k^{-(r+3)}),
\end{eqnarray*}
where $M_i=M_i(\eta_1, \ldots, \eta_i)$ and $G_r=G_r(\eta_1, \ldots, \eta_r)$ are functions determined by the previous data, and the absolute of the term $O(k^{-(r+3)})$ is bounded by $A \cdot k^{-(r+3)}$ for some constant $A$ which only depends on
\[
\max_{X \times [0, T]} \left\{ \left| \frac{\partial^{r+3} \phi_t}{\partial t^{r+3}} \right|, \left| \frac{\partial^{r+2} \eta_1}{\partial t^{r+2}} \right|, \ldots, \left| \frac{\partial^2 \eta_{r+1}}{\partial t^2} \right| \right\}.
\]
Hence $\eta_{r+1}$ only affects the $O(k^{-(r+2)})$-term in the above expansion and no lower-order terms of $k^{-1}$. The contribution of $\eta_{r+1}$ to the coefficient of $k^{-(r+2)}$ is just  $\frac{\partial \eta_{r+1}}{\partial t} (t)$.

Let $\Delta_t$ be the (negative) $\bar{\partial}$-Laplacian with respect to $\omega_{\phi_t}$. Since the linearized operator of ${\rm MA}/\mu_0$ at $\phi_t$ is computed as
\[
\frac{d}{ds} \left. \left( \frac{{\rm MA}(\phi_t+sf)}{\mu_0} \right) \right|_{s=0} = (\Delta_t f)\cdot \frac{{\rm MA}(\phi_t)}{\mu_0},
\]
we have
\begin{eqnarray*}
\frac{{\rm MA}(\widetilde{\phi}_t^{(r+1)})}{\mu_0} &=& \frac{{\rm MA}(\phi_t)}{\mu_0} + \sum_{i=1}^{r} k^{-i} P_i(t) +k^{-(r+1)} \left( (\Delta_t \eta_{r+1}) \cdot \frac{{\rm MA}(\phi_t)}{\mu_0} + Q_r(t) \right) \\
&\hbox{}& + O(k^{-(r+2)}),
\end{eqnarray*}
where $P_i=P_i(\eta_1, \ldots, \eta_i)$, and $Q_r=Q_r(\eta_1, \ldots, \eta_r)$.  Since we have the uniform $C^{\infty}$-convergence $\widetilde{\phi}_t^{(r)} \to \phi_t$ as $k \to \infty$, the metrics $\widetilde{\phi}_t^{(r)}$ are in a $C^{\infty}$-bounded set as $t$ varies in $[0,T]$. Thus we can apply Proposition \ref{abf} and find that $\bar{b}_i(\widetilde{\phi}_t^{(r+1)})$ are polynomials in the curvature of $\widetilde{\phi}_t^{(r+1)}$ and $\mu_0/{\rm MA}(\widetilde{\phi}_t^{(r+1)})$. It follows that $\bar{b}_i(\widetilde{\phi}_t^{(r+1)})=\bar{b}_i(\widetilde{\phi}_t^{(r)})+O(k^{-(r+1)})$, in particular, the first contribution of $\eta_{r+1}$ to $\bar{b}_i(\widetilde{\phi}_t^{(r+1)})$ occurs at $O(k^{-(r+1)})$. Thus we have
\begin{eqnarray*}
F_{\mu_0}^{(k)}(\widetilde{\phi}_t^{(r+1)}) &=& \frac{1}{k} \log \frac{{\rm MA}(\widetilde{\phi}_t^{(r+1)})}{\mu_0} \\
&\hbox{}& + \frac{1}{k} \log \left(1+ k^{-1} \bar{b}_1(\widetilde{\phi}_t^{(r+1)}) + \cdots + k^{-(r+1)} \bar{b}_{r+1} (\widetilde{\phi}_t^{(r+1)}) +O(k^{-(r+2)}) \right) \\
&=& \frac{1}{k} \log \frac{{\rm MA}(\phi_t)}{\mu_0} + \sum_{i=1}^{r} k^{-(i+1)} R_i(t) \\
&\hbox{}& + k^{-(r+2)} (\Delta_t \eta_{r+1}(t) + S_r(t)) + O(k^{-(r+3)}),
\end{eqnarray*}
where $R_i=R_i (\eta_1, \ldots, \eta_i)$ and $S_r=S_r(\eta_1, \ldots, \eta_r)$. Putting $T_r(t):=S_r(t)-G_r(t)$, we obtain
\[
(\widetilde{\phi}_{t+1/k}^{(r+1)} - \widetilde{\phi}_t^{(r+1)}) - F_{\mu_0}^{(k)}(\widetilde{\phi}_t^{(r+1)})=k^{-(r+2)} \left( \frac{\partial \eta_{r+1}}{\partial t} (t) -\Delta_t \eta_{r+1}(t)-T_r (t) \right) + O(k^{-(r+3)}),
\]
where we used the induction hypothesis \eqref{clm}. Hence we may choose $\eta_{r+1}$ as a solution of the linear, parabolic PDE:
\begin{equation} \label{lpe}
\begin{cases}
\frac{\partial \eta_{r+1}}{\partial t} (t) - \Delta_t \eta_{r+1}(t) =T_r (t)  \\
\eta_{r+1}(0)=0,
\end{cases}
\end{equation}
where $T_r=T_r (\eta_1, \ldots, \eta_r)$ is determined in the previous process. The equation \eqref{lpe} is a linear, inhomogeneous heat equation. Since the spectra of $-\Delta_t$ is bounded below, $\Delta_t$ generates a strongly continuous analytic semigroup for each $t$. Hence there exists a unique long time solution of \eqref{lpe} (where we used general results in the semigroup theory, for instance, see \cite[Section 1.2]{Ama95}).

\vspace{2mm}
\noindent
{\it The setting} $({\bf S}_{\pm})$. We can compute the term $\widetilde{\phi}_{t+1/k}^{(r+1)} - \widetilde{\phi}_t^{(r+1)}$ as with the case $({\bf S}_0)$. The difference in computations only comes from the term $F_{\mu_{\pm}}^{(k)}(\widetilde{\phi}_t^{(r+1)})$, more precisely, the linearization of the operator ${\rm MA}/\mu_{\pm}$:
\[
\frac{d}{ds} \left. \left( \frac{{\rm MA}(\phi_t+sf)}{\mu_{\pm}(\phi_t+sf)} \right) \right|_{s=0} = (\Delta_t f \mp f)\cdot \frac{{\rm MA}(\phi_t)}{\mu_{\pm}(\phi_t)}.
\]
Using this, we obtain
\begin{eqnarray*}
F_{\mu_{\pm}}^{(k)}(\widetilde{\phi}_t^{(r+1)}) &=& \frac{1}{k} \log \frac{{\rm MA}(\phi_t)}{\mu_{\pm}(\phi_t)} + \sum_{i=1}^{r} k^{-(i+1)} R_i(t) \\
&\hbox{}& + k^{-(r+2)} (\Delta_t \eta_{r+1}(t) \mp \eta_{r+1}(t)+ S_r(t)) + O(k^{-(r+3)}),
\end{eqnarray*}
where $R_i=R_i (\eta_1, \ldots, \eta_i)$ and $S_r=S_r(\eta_1, \ldots, \eta_r)$. Thus we may take $\eta_{r+1}$ as the solution of the linear, parabolic PDE:
\begin{equation} \label{lpf}
\begin{cases}
\frac{\partial \eta_{r+1}}{\partial t} (t) - \Delta_t \eta_{r+1}(t) \pm \eta_{r+1} (t) =T_r (t)\\
\eta_{r+1}(0)=0,
\end{cases}
\end{equation}
where $T_r$ is a function depending only on $\eta_1, \ldots, \eta_r$. Since the operator $-\Delta_t \pm {\rm Id}$ have only finitely many negative eigenvalues, the spectra of the operator $-\Delta_t \pm {\rm Id}$ are bounded below. Hence the operator $\Delta_t \mp {\rm Id}$ generates a strongly continuous analytic semigroup for each $t$, and we can use again the result \cite[Section 1.2]{Ama95} to obtain the long time solution $\eta_{r+1}$.
\end{proof}
Now we return to the proof of Lemma \ref{hoa}.

\vspace{2mm}
\noindent
{\it The setting} $({\bf S}_0)$. We show that the equation
\begin{equation} \label{cye}
d \left( \phi_k^{(m)},\widetilde{\phi}_{m/k}^{(r)} \right) \leq C \cdot \frac{m}{k^{r+2}}
\end{equation}
holds as long as $m/k \leq T$, where the functions $\eta_1, \ldots, \eta_r$ and the constant $C>0$ are determined in the previous claim.

We prove this by induction of $m$. Notice that the case $m=0$ is trivial. Assume that the equation \eqref{cye} holds for $m$, and let $k$ be any integer such that $\frac{m+1}{k} \leq T$. Applying \eqref{clm} with $t:=\frac{m}{k} \leq \frac{m+1}{k} \leq T$ yields
\[
\sup_X \left| \widetilde{\phi}_{(m+1)/k}^{(r)} - \widetilde{\phi}_{m/k}^{(r)} - F_{\mu_0}^{(k)} (\widetilde{\phi}_{m/k}^{(r)}) \right| \leq C \cdot \frac{1}{k^{r+2}}.
\]
On the other hand, using Proposition \ref{mit} (1), we have
\begin{eqnarray*}
\sup_X \left| (\widetilde{\phi}_{m/k}^{(r)} + F_{\mu_0}^{(k)} (\widetilde{\phi}_{m/k}^{(r)}))-\phi_{k}^{(m+1)} \right| &=& \sup_X \left| (\widetilde{\phi}_{m/k}^{(r)} + F_{\mu_0}^{(k)} (\widetilde{\phi}_{m/k}^{(r)}))-(\phi_k^{(m)}+F_{\mu_0}^{(k)}(\phi_k^{(m)})) \right| \\
&=& d \left({\mathcal T}_{k,\mu_0}(\widetilde{\phi}_{m/k}^{(r)}), {\mathcal T}_{k,\mu_0}(\phi_k^{(m)}) \right) \\
&\leq& d \left(\widetilde{\phi}_{m/k}^{(r)}, \phi_k^{(m)} \right) \\
&\leq& C \cdot \frac{m}{k^{r+2}},
\end{eqnarray*}
where we used the induction hypothesis in the last inequality. Combining these two inequalities, we have
\begin{eqnarray*}
&\hbox{}& d \left( \phi_{k}^{(m+1)},\widetilde{\phi}_{(m+1)/k}^{(r)} \right) \\
&\hbox{}& \leq \sup_X \left| \phi_{k}^{(m+1)}-\widetilde{\phi}_{m/k}^{(r)}-F_{\mu_0}^{(k)}(\widetilde{\phi}_{m/k}^{(r)}) \right|+ \sup_X \left| \widetilde{\phi}_{m/k}^{(r)}+F_{\mu_0}^{(k)}(\widetilde{\phi}_{m/k}^{(r)})-\widetilde{\phi}_{(m+1)/k}^{(r)} \right| \\
&\hbox{}& \leq C \cdot \frac{m}{k^{r+2}} + C \cdot \frac{1}{k^{r+2}} \\
&\hbox{}& = C \cdot \frac{m+1}{k^{r+2}}.
\end{eqnarray*}
Hence the equation \eqref{cye} holds for $m+1$.

From \eqref{cye}, we have
\[
d \left( \phi_k^{(m)},\widetilde{\phi}_{m/k}^{(r)} \right) \leq C \cdot \frac{m}{k^{r+2}} \leq CT \cdot \frac{1}{k^{r+1}}.
\]
Finally, replacing $CT$ with $C$, we obtain the desired result.

\vspace{2mm}
\noindent
{\it The setting} $({\bf S}_+)$. Thanks to Proposition \ref{mit} (2), the distance defined by the sup-norm is decreasing along the iteration. Hence the proof for the setting $({\bf S}_0)$ carries over essentially verbatim to this case.

\vspace{2mm}
\noindent
{\it The setting} $({\bf S}_-)$. We show that the equation
\begin{equation} \label{sp3}
d \left( \phi_k^{(m)},\widetilde{\phi}_{m/k}^{(r)} \right) \leq C \left( 1+\frac{1}{k}\right)^m \cdot \frac{m}{k^{r+2}}
\end{equation}
holds as long as $m/k \leq T$, where the functions $\eta_1, \ldots, \eta_r$ and the constant $C>0$ are determined in the previous claim.

The case $m=0$ is trivial. We assume that the equation \eqref{sp3} holds for $m$ and let $k$ be any integer such that $\frac{m+1}{k} \leq T$. Applying \eqref{clm} with $t:=\frac{m}{k} \leq \frac{m+1}{k} \leq T$ yields
\[
\sup_X \left| \widetilde{\phi}_{(m+1)/k}^{(r)} - \widetilde{\phi}_{m/k}^{(r)} - F_{\mu_-}^{(k)} (\widetilde{\phi}_{m/k}^{(r)}) \right| \leq C \cdot \frac{1}{k^{r+2}}.
\]
On the other hand, using Proposition \ref{mit} (2) and the induction hypothesis, we have
\begin{eqnarray*}
\sup_X \left| (\widetilde{\phi}_{m/k}^{(r)} + F_{\mu_-}^{(k)} (\widetilde{\phi}_{m/k}^{(r)}))-\phi_{k}^{(m+1)} \right| &=& \sup_X \left| (\widetilde{\phi}_{m/k}^{(r)} + F_{\mu_-}^{(k)} (\widetilde{\phi}_{m/k}^{(r)}))-(\phi_k^{(m)}+F_{\mu_-}^{(k)}(\phi_k^{(m)})) \right| \\
&\leq& \left(1+\frac{1}{k} \right) \cdot d \left( \widetilde{\phi}_{m/k}^{(r)}, \phi_k^{(m)} \right) \\
&\leq& C \left( 1+\frac{1}{k} \right)^{m+1} \cdot \frac{m}{k^{r+2}}.
\end{eqnarray*}
Combining these two inequalities, we have
\begin{eqnarray*}
d \left( \phi_{k}^{(m+1)},\widetilde{\phi}_{(m+1)/k}^{(r)} \right) &\leq& C \left(1+\frac{1}{k} \right)^{m+1} \cdot \frac{m}{k^{r+2}} + C \cdot \frac{1}{k^{r+2}} \\
&\leq& C \left(1+\frac{1}{k} \right)^{m+1} \cdot \frac{m}{k^{r+2}} + C \left(1+\frac{1}{k} \right)^{m+1} \cdot \frac{1}{k^{r+2}} \\
&=& C \left(1+\frac{1}{k} \right)^{m+1} \cdot \frac{m+1}{k^{r+2}}.
\end{eqnarray*}
Hence the equation \eqref{sp3} holds for $m+1$.

Since
\[
\left(1+\frac{1}{k} \right)^m=\left( \left(1+\frac{1}{k} \right)^k \right)^{m/k} \leq e^{m/k} \leq e^T,
\]
combining with \eqref{sp3} implies
\[
d \left( \phi_k^{(m)},\widetilde{\phi}_{m/k}^{(r)} \right) \leq C e^T \cdot \frac{m}{k^{r+2}} \leq CT e^T \cdot \frac{1}{k^{r+1}}.
\]
We accomplishes the proof by replacing $CTe^T$ with $C$.  
\end{proof}
In order to apply projective estimates, we have to approximate the perturbed flow $\widetilde{\phi}_t^ {(r)}$ by a smooth family of Bergman metrics.
\begin{lemma} \label{abm}
For any integer $r$, there exists a smooth family of metrics $\overline{\phi}_t^{(r)}$ ($t \in [0, \infty)$) in ${\mathcal H}(L)$ such that
\begin{enumerate}
\item ${\mathcal T}_k (\overline{\phi}_t^{(r)}) = \widetilde{\phi}_t^ {(r)} + O(k^{-(r+1)})$ as $k \to \infty$.
\item $\overline{\phi}_t^{(r)} \to \phi_t$ in the $C^{\infty}$-topology as $k \to \infty$.
\item ${\mathcal T}_k (\overline{\phi}_t^{(r)}) \to \phi_t$ in the $C^{\infty}$-topology as $k \to \infty$.
\end{enumerate}
Moreover, all of the above properties hold uniformly in $t \in [0, T]$.
\end{lemma}
\begin{proof}
We can prove this by the same argument in \cite[Section 5]{Don09}. Since $\widetilde{\phi}_t^{(r)} \to \phi_t$ in the $C^{\infty}$-topology as $k \to \infty$, uniformity of the asymptotic expansion (cf. Proposition \ref{abf}) yields
\[
F^{(k)} (\widetilde{\phi}_t^{(r)}) = e_2(t) k^{-2} + O(k^{-3}),
\]
for some smooth function $e_2(t)$. Hence ${\mathcal T}_k (\widetilde{\phi}_t^{(r)})=\widetilde{\phi}_t^{(r)} + O(k^{-2})$. If we put $\overline{\phi}_t^{(r)}:=\widetilde{\phi}_t^{(r)} - k^{-2} e_2(t)$, then we have
\begin{eqnarray*}
{\mathcal T}_k (\overline{\phi}_t^{(r)}) &=& \overline{\phi}_t^{(r)} + F^{(k)} (\overline{\phi}_t^{(r)}) \\
&=& \widetilde{\phi}_t^{(r)} +\left( F^{(k)} (\overline{\phi}_t^{(r)}) - k^{-2} e_2(t) \right) \\
&=& \widetilde{\phi}_t^{(r)} +O(k^{-3}).
\end{eqnarray*}
We can repeat this process and obtain functions $e_2, \ldots, e_r$ (depending smoothly on $t$) such that ${\mathcal T}_k (\overline{\phi}_t^{(r)}) = \widetilde{\phi}_t^ {(r)} + O(k^{-(r+1)})$ with $\overline{\phi}_t^{(r)}:=\widetilde{\phi}_t^{(r)} - \sum_{i=2}^{r} k^{-i} e_i(t)$, which completes the proof of (1). By the definition of $\overline{\phi}_t^{(r)}$, we have $\left\| \overline{\phi}_t^{(r)} -\widetilde{\phi}_t^{(r)} \right\|_{C^l}= \left\| \sum_{i=2}^{r} k^{-i} e_i(t) \right\|_{C^l}  \to 0$ as $k \to \infty$ for any non-negative integer $l$. Combining with Remark \ref{frm}, we have (2). From (2) and uniformity of the asymptotic expansion for Bergman function, we have (3). Finally, we stress that all of the above arguments hold uniformly in $t$: we have used asymptotic expansions for $\widetilde{\phi}_t^ {(r)}$ and $\overline{\phi}_t^{(r)}$. When $t$ varies in a compact interval $[0,T]$, these family of metrics stay in a $C^{\infty}$-bounded set since there are only finitely many perturbations $\eta_i$, $e_i$ present. Thus these asymptotic expansions are uniform in $t$.
\end{proof}
Now we set $\widehat{\phi}_t^{(r)}:={\mathcal T}_k (\overline{\phi}_t^{(r)})$ and $\widehat{H}_t^{(r)}:= {\rm Hilb}_k (\overline{\phi}_t^{(r)}) \in {\mathcal B}_k$ for the corresponding hermitian form, i.e., $\widehat{\phi}_t^{(r)}={\rm FS}_k (\widehat{H}_t^{(r)})$.
By Lemma \ref{hoa} and Lemma \ref{abm} (1), we obtain:
\begin{lemma} \label{sre}
\[
d \left( \phi_k^{(m)}, \widehat{\phi}_{m/k}^{(r)} \right)=O(k^{-(r+1)})
\]
for any $(k,m)$ such that $m/k \leq T$.
\end{lemma}

\subsection{Distance function $d_k$ on ${\mathcal B}_k$}
Let $d_k$ be the distance function arising from the Riemannian structure ${\rm tr}_H( \delta H, \delta H):={\rm tr}(\delta H \cdot H^{-1} \cdot \delta H \cdot H^{-1})$ on ${\mathcal B}_k$. For $m \geq 1$, we denote the hermitian norm which corresponds to $\phi_k^{(m)}$ by $H_k^{(m)}$, i.e., $\phi_k^{(m)}={\rm FS}_k (H_k^{(m)})$. We prove that the higher order estimate of the distance $d \left( \phi_k^{(m)},\widehat{\phi}_{m/k}^{(r)} \right)$ yields the estimate of $d_k \left( H_k^{(m)}, \widehat{H}_{m/k}^{(r)} \right)$.
\begin{lemma} \label{bdf}
If $r>2n$, we have $d_k \left(H_k^{(m)}, \widehat{H}_{m/k}^{(r)} \right) = O(k^{2n-r})$.
\end{lemma}
\begin{proof}
Let $(s_i)$ be an $\widehat{H}_{m/k}^{(r)}$-orthonormal and an $H_k^{(m)}$-orthogonal basis. Then we can find $\lambda_j \in {\mathbb R}$ so that $(e^{\lambda_j}s_j)$ is an $H_k^{(m)}$-orthonormal basis. Then the distance between these two hermitian forms is computed as
\[
d_k \left( H_k^{(m)}, \widehat{H}_{m/k}^{(r)} \right)= \sqrt{\sum_{j=1}^{N_k} |\lambda_j|^2}.
\]
We define the function $f$ by $\phi_k^{(m)}-\widehat{\phi}_{m/k}^{(r)}=\frac{1}{k} \log(1+f)$. Now we apply the argument in \cite[Lemma 2.18]{Has15}. The direct computation shows that
\begin{eqnarray*}
\log(1+f)&=&k\left(\phi_k^{(m)}-\widehat{\phi}_{m/k}^{(r)} \right)\\
&=& k \left( {\rm FS}_k(H_k^{(m)})-{\rm FS}_k(\widehat{H}_{m/k}^{(r)}) \right)\\
&=& \log \frac{\sum_{j=1}^{N_k} e^{2 \lambda_j}|s_j|^2}{\sum_{j=1}^{N_k}|s_j|^2}.
\end{eqnarray*}
Hence, for any $\phi \in {\mathcal H}(L)$, we have
\begin{equation} \label{hap}
(1+f) \cdot \sum_{i=1}^{N_k}|s_j|^2e^{-k\phi}=\sum_{j=1}^{N_k}e^{2\lambda_j}|s_j|^2e^{-k\phi}.
\end{equation}
Since the map ${\rm Hilb}_k$ is surjective, we can choose a weight $\phi_i \in {\mathcal H}(L)$ so that $(e^{k/2}s_1,\ldots,e^{k/2}s_{i-1},s_i,e^{k/2}s_{i+1},\ldots,e^{k/2}s_{N_k})$ is a ${\rm Hilb}_k(\phi_i)$-ONB. We set
\begin{eqnarray*}
v_i&:=&(\left\|s_1\right\|_{{\rm Hilb}_k(\phi_i)}^2,\ldots,\left\|s_{N_k}\right\|_{{\rm Hilb}_k(\phi_i)}^2)\\
&=&(e^{-k},\ldots,e^{-k},\begin{array}[b]{c}i\\[-2pt] \vee \\[-2pt]1\end{array},e^{-k},\ldots,e^{-k}),
\end{eqnarray*}
\[
A=
\begin{pmatrix}
v_1\\
\vdots\\
v_{N_k}
\end{pmatrix}
,\;\;
F=(F_{i,j})=\left( \int_X f|s_j|^2e^{-k\phi_i} {\rm MA}(\phi_i) \right).
\]
Then we find that $\|A\|_{\rm op} \leq 2$ and $\|A^{-1}\|_{\rm op} \leq 2$ if $k$ is sufficiently large. Moreover, if we put $\phi=\phi_i$ in \eqref{hap}, then we have
\[
(A+F)
\begin{pmatrix}
1\\
\vdots\\
1
\end{pmatrix}
=A
\begin{pmatrix}
e^{2\lambda_1}\\
\vdots\\
e^{2\lambda_{N_k}}
\end{pmatrix}
,
\]
and hence
\begin{equation} \label{leq}
\begin{pmatrix}
e^{2\lambda_1}-1\\
\vdots\\
e^{2\lambda_{N_k}}-1
\end{pmatrix}
=A^{-1}F
\begin{pmatrix}
1\\
\vdots\\
1
\end{pmatrix}
.
\end{equation}
On the other hand, since
\begin{eqnarray*}
\|F\|_{\rm max}&:=&\max_{i,j}\{|F_{i,j}|\} \\
&\leq& \sup_X|f| \cdot \max_{i,j} \left\{ \int_X|s_j|^2e^{-k\phi_i} {\rm MA}(\phi_i) \right\} \\
&=& \sup_X|f| \cdot \max_{i,j} \|s_j\|_{{\rm Hilb}_k(\phi_i)}^2 \\
&=& \sup_X|f|,
\end{eqnarray*}
we obtain
\[
\|A^{-1}F\|_{\rm op} \leq \|A^{-1}\|_{\rm op} \|F\|_{\rm op} \leq 2 \|F\|_{\rm HS} \leq 2N_k\|F\|_{\rm max} \leq 2N_k \sup_X|f|,
\]
where $\|F\|_{\rm HS}:=\sqrt{\sum_{i,j=1}^{N_k}|F_{i,j}|^2}$ is the Hilbert-Schmidt norm of $F$.
Combining with \eqref{leq}, we find that
\[
2N_k \sup_X|f| \geq \|A^{-1}F\|_{\rm op}=\sup_{x \neq 0} \frac{\|(A^{-1}F)(x)\|}{\|x\|} \geq N_k^{-1/2} \sqrt{\sum_{j=1}^{N_k}|e^{2\lambda_j}-1|^2}.
\]
Thus we have
\[
1-2N_k^{3/2}\sup_X|f| \leq e^{2\lambda_j}\leq1+2N_k^{3/2} \sup_X|f|.
\]
Now we assume $r>2n$. Then, by Lemma \ref{sre}, we see that $\sup_X|f|=O(k^{-r})$ and ${N_k}^{3/2}\sup_X|f|=O(k^{\frac{3}{2}n-r})$ (where we used $N_k=O(k^n)$). Thus if $k$ is sufficiently large, we can take the $\log$ of the above equation and know that
\[
\frac{1}{2} \log \left(1-2{N_k}^{3/2} \sup_X |f| \right) \leq \lambda_j \leq \frac{1}{2} \log \left(1+2{N_k}^{3/2} \sup_X|f| \right),
\]
where $\log \left(1-2{N_k}^{3/2} \sup_X |f| \right)=O(k^{\frac{3}{2}n-r})$ and $\log \left(1+2{N_k}^{3/2} \sup_X|f| \right)=O(k^{\frac{3}{2}n-r})$. Hence we have $|\lambda_j|^2=O(k^{3n-2r})$ and 
\[
d_k \left( H_k^{(m)}, \widehat{H}_{m/k}^{(r)} \right) \leq \sqrt{N_k \max_j |\lambda_j|^2}\\ = O(k^{2n-r}).
\]
\end{proof}
\subsection{Operator norm $\|\bar{\mu}( \cdot )\|_{\rm op}$}
For $H \in {\mathcal B}_k$, let $\sqrt{-1} \mu(H)$ be the moment map of the action of the corresponding unitary group. If we take an $H$-orthonormal basis $(s_i)$, $\mu(H)$ can be represented as a matrix-valued function
\[
(\mu(H))_{\alpha, \beta}=\frac{(s_{\alpha}, s_{\beta})}{\sum_{i=1}^{N_k} |s_i|^2}.
\]
We are interested in the {\it center of mass} $\bar{\mu}(H)$:
\[
\bar{\mu}(H):=k^n \int_X \mu(H) {\rm MA}({\rm FS}_k(H)).
\]
The following Lemma is a direct consequence from Lemma \ref{abm} (2) and \cite[Lemma 15, Remark 16]{Fin10}:
\begin{lemma} \label{bop}
$\left\| \bar{\mu}(\widehat{H}_t^{(r)}) -{\rm Id} \right\|_{\rm op} \to 0$ uniformly for $t \in [0, T]$.
\end{lemma}
Let $H(s)$ ($s \in [0,1]$) be a Bergman geodesic with $H(0)=\widehat{H}_{m/k}^{(r)}$ and $H(1)=H_k^{(m)}$. 
The next lemma shows that the distance $d_k$ controls the operator norm $\|\bar{\mu}(\cdot)\|_{\rm op}$.
\begin{lemma} \label{bon}
If $r>2n$, the operator norm $\|\bar{\mu}(H(s))\|_{\rm op}$ is uniformly bounded for any $(k,m)$ such that $m/k \leq T$ and $s \in [0,1]$.
\end{lemma}
\begin{proof}
By the previous lemma, the operator norm $\left\| \bar{\mu}(\widehat{H}_{m/k}^{(r)}) \right\|_{\rm op}$ is uniformly bounded as long as $m/k \in [0, T]$. Applying \cite[Proposition 24]{Fin10} to our case, we obtain
\begin{eqnarray*}
\|\bar{\mu}(H(s))\|_{\rm op} &\leq& \exp \left( 2 d_k \left( H(s), \widehat{H}_{m/k}^{(r)} \right) \right) \cdot \left\|\bar{\mu}(\widehat{H}_{m/k}^{(r)}) \right\|_{\rm op} \\
&\leq& \exp \left( 2 d_k \left( H_k^{(m)}, \widehat{H}_{m/k}^{(r)} \right) \right) \cdot \left\|\bar{\mu}(\widehat{H}_{m/k}^{(r)}) \right\|_{\rm op}.
\end{eqnarray*}
Hence $\|\bar{\mu}(H(s))\|_{\rm op}$ is bounded as long as $r>2n$ by Lemma \ref{bdf}.
\end{proof}
\subsection{Bounded geometry}
In this section, we review the definitions of $R$-bounded geometry and several related results in \cite[Section 4]{Fin10}. We use the {\it large} K\"ahler metrics in the class $kc_1(L)$ to avoid worrying about powers of $k$. We fix a reference K\"ahler metric $\omega_0 \in c_1(L)$ and denote a large reference K\"ahler metric $\widetilde{\omega}_0:=k\omega_0 \in k c_1(L)$.
\begin{definition}
We say that $\widetilde{\omega} \in k c_1(L)$ has $R$-bounded geometry in $C^l$ if $\widetilde{\omega} > R^{-1} \widetilde{\omega}_0$ and
\[
\|\widetilde{\omega}-\widetilde{\omega}_0\|_{C^l} < R,
\]
where the norm $\| \cdot \|_{C^l}$ is that determined by the metric $\widetilde{\omega}_0$.
\end{definition}
For a family of metrics which has $R$-bounded geometry in $C^l$, we can control the $C^{l-2}$-norm of K\"ahler metrics by means of geometric data in the Bergman space ${\mathcal B}_k$.
\begin{lemma}[Lemma 13, \cite{Fin10}] \label{ccp}
Let $H(s)$ be a smooth path in ${\mathcal B}_k$. If $\widetilde{\omega}(s)= \omega_{k{\rm FS}_k (H(s))}$ for $s \in [0,1]$ have $R$-bounded geometry in $C^l$, and $\|\bar{\mu}(H(s))\|_{\rm op} < K$, then
\[
\|\widetilde{\omega}(0)-\widetilde{\omega}(1)\|_{C^{l-2}} < C K L,
\]
where the constant $C>0$ depends only on $R$ and $l$, but not on $k$, and $L$ is the length of the path $H(s)$ ($s \in [0,1]$).
\end{lemma}
The next lemma is also useful to check the condition for bounded geometry.
\begin{lemma}[Lemma 14, \cite{Fin10}] \label{bcd}
Let $H_k \in {\mathcal B}_k$ be a sequence of metrics such that the corresponding sequence of metrics $\widetilde{\omega}_k:= \omega_{k{\rm FS}_k(H_k)}$ has $R/2$-bounded geometry in $C^{l+2}$ and such that $\|\bar{\mu}(H_k)\|_{\rm op}$ is uniformly bounded. Then there is a constant $C>0$ (which depends only on $R$ and $l$, but not on $k$) such that if $H \in {\mathcal B}_k$ satisfies $d_k(H_k, H) < C$, then the corresponding metric $\widetilde{\omega}:= \omega_{k{\rm FS}_k(H)}$ has $R$-bounded geometry in $C^l$.
\end{lemma}
\section{Proof of Theorem \ref{thm}}
Now we give a proof of Theorem \ref{thm}.
\begin{proof}[Proof of Theorem \ref{thm}]
We use $\omega_{\phi_t}$ as our reference metric. Since we have the $C^{\infty}$-convergence $\widehat{\phi}_{m/k}^{(r)} \to \phi_t$ in the double scaling limit $m/k \to t$ (cf. Lemma \ref{abm} (3)), we know that the metrics $\omega_{\widehat{\phi}_{m/k}^{(r)}}$ has $R/2$-bounded geometry in $C^{l+4}$. If we take $r>2n$, we have $d_k \left( H_k^{(m)}, \widehat{H}_{m/k}^{(r)} \right) \to 0$ (cf. Lemma \ref{bdf}) and $\left\| \bar{\mu}(\widehat{H}_{m/k}^{(r)}) \right\|_{\rm op}$ is uniformly bounded (cf. Lemma \ref{bop}). 
Hence we can apply Lemma \ref{bcd} to our case and find that a family of metrics $\omega_{FS_k(H(s))}$ ($s \in [0,1]$) has $R$-bounded geometry in $C^{l+2}$. Combining with Lemma \ref{bdf}, Lemma \ref{bon} and Lemma \ref{ccp}, we obtain
\[
\left\| k\omega_{\phi_k^{(m)}}-k\omega_{\widehat{\phi}_{m/k}^{(r)}} \right\|_{C^l (k\omega_{\phi_t})} \leq CK \cdot d_k \left( H_k^{(m)}, \widehat{H}_{m/k}^{(r)} \right)
\leq CKM \cdot k^{2n-r},
\]
where the constant $M>0$ depends only on $T$, and $K$ is the uniform upper bound of $\|\bar{\mu}(H(s))\|_{\rm op}$. Rescaling the inequality, we have
\begin{equation} \label{sei}
\left\| \omega_{\phi_k^{(m)}}-\omega_{\widehat{\phi}_{m/k}^{(r)}} \right\|_{C^l(\omega_{\phi_t})} \leq CKM \cdot k^{\frac{l}{2}+2n-r}.
\end{equation}
Hence, if we take $r$ so that $r>\frac{l}{2}+2n$ and take the limit $k \to \infty$, we have the $C^l$-convergence
\[
\left\|\omega_{\phi_k^{(m)}} - \omega_{\phi_t}\right\|_{C^l(\omega_{\phi_t})} \leq \left\|\omega_{\phi_k^{(m)}} - \omega_{\widehat{\phi}_{m/k}^{(r)}} \right\|_{C^l(\omega_{\phi_t})}+ \left\|\omega_{\widehat{\phi}_{m/k}^{(r)}}-\omega_{\phi_t} \right\|_{C^l(\omega_{\phi_t})} \to 0
\]
as desired. Finally, Remark \ref{frm} together with Lemma \ref{abm} (1) implies 
\[
\left\|\omega_{\widehat{\phi}_{m/k}} - \omega_{\phi_t}\right\|_{C^l(\omega_{\phi_t})}=O(k^{-1}).
\]
On the other hand, if we take $r$ so that $r>\frac{l}{2}+2n+\frac{1}{2}$ in \eqref{sei}, we get
\[
\left\| \omega_{\phi_k^{(m)}}-\omega_{\widehat{\phi}_{m/k}^{(r)}} \right\|_{C^l(\omega_{\phi_t})}=O(k^{-1}).
\]
Combining these two estimates, we obtain the speed of convergence
\[
\left\|\omega_{\phi_k^{(m)}} - \omega_{\phi_t}\right\|_{C^l(\omega_{\phi_t})}=O(k^{-1}).
\]
This completes the proof of Theorem \ref{thm}.
\end{proof}

\end{document}